\documentclass{amsart} 
 
\usepackage{xcolor}

\usepackage{mathrsfs}

\usepackage{enumitem}

\usepackage{amsthm, amsmath, amssymb, amsrefs} 

\usepackage{caption, chngcntr, setspace} 

\usepackage{bm, dsfont, mathrsfs, mathtools, mhsetup}

\usepackage{cases}

\usepackage{graphicx, tikz, tikz-cd, pgfplots, wrapfig} 

\usepackage{textcomp, microtype, csquotes, wasysym, stmaryrd}

\usepackage{array, multicol, multirow}

\usepackage{forloop, ragged2e, verbatim} 

\usepackage{hyperref, cleveref} 

\usepackage{tikzsymbols}

\hypersetup{
    colorlinks=true,
    linkcolor=blue,
    filecolor=magenta,
    urlcolor=cyan,
}

\newcommand{\lpcritical}{\rho_*}
\newcommand{\crist}{\lpcritical}

\newtheorem{thm}{Theorem}[section]

\newtheorem{lemma}[thm]{Lemma}
\newtheorem{cor}[thm]{Corollary}

\newtheorem{dfn}[thm]{Definition}

\theoremstyle{remark}

\definecolor{cornflowerblue}{rgb}{0.39, 0.58, 0.93}
\definecolor{yaleblue}{rgb}{0.06, 0.3, 0.57}
\definecolor{blizzardblue}{rgb}{0.67, 0.9, 0.93}
\definecolor{unitednationsblue}{rgb}{0.36, 0.57, 0.9}
\definecolor{carolinablue}{rgb}{0.6, 0.73, 0.89}
\definecolor{columbiablue}{rgb}{0.61, 0.87, 1.0}
\definecolor{babyblue}{rgb}{0.54, 0.81, 0.94}
\definecolor{darkcerulean}{rgb}{0.03, 0.27, 0.49}

\definecolor{mountainmeadow}{rgb}{0.19, 0.73, 0.56}
\definecolor{pastelgreen}{rgb}{0.47, 0.87, 0.47}
\definecolor{celadon}{rgb}{0.67, 0.88, 0.69}
\definecolor{teagreen}{rgb}{0.82, 0.94, 0.75}

\definecolor{ashgrey}{rgb}{0.7, 0.75, 0.71}
\definecolor{corn}{rgb}{0.98, 0.93, 0.36}
\definecolor{canaryyellow}{rgb}{1.0, 0.94, 0.0}

\definecolor{richlavender}{rgb}{0.67, 0.38, 0.8}
\definecolor{americanrose}{rgb}{1.0, 0.01, 0.24}
\definecolor{palepink}{rgb}{0.98, 0.85, 0.87}

\newcommand{\bP}{\mathbb P}

\newcommand{\bZ}{\mathbb Z}
\newcommand{\Z}{\mathbb Z}
\newcommand{\fixation}{\bf\texttt{Fixation}}
\newcommand{\Instr}{\mathsf{Instr}}
\newcommand{\Left}{\texttt{Left}}
\newcommand{\Right}{\texttt{Right}}
\newcommand{\Sleep}{\texttt{Sleep}}
\newcommand{\positive}{\text{\bf\texttt{Positive Odometer}}}
\newcommand{\tilt}{\text{\bf\texttt{Tilt}}}
\newcommand{\s}{\mathfrak{s}}
\newcommand{\omegatilde}{\tilde{\omega}}
\newcommand{\omegaprime}{\omega_{\text{initial}}}
\newcommand{\omegadouble}{\omega_{\text{final}}}

\newcommand{\init}{{\omega^\ast}}
\newcommand{\plentiful}{{\bf\texttt{plentiful}}}
\newcommand{\sparse}{{\bf\texttt{sparse}}}

\pgfplotsset{compat=1.15}

\begin{document}
\title{Activated Random Walks on $\bZ$ with Critical Particle Density} 
\author[Madeline~Brown]{ \ Madeline~Brown} 
	\address{Department of Mathematics, University of Washington, Seattle, WA 98195} 
	\email{madbro@uw.edu} 
\author[Christopher~Hoffman]{ \ Christopher~Hoffman} 
	\address{Department of Mathematics, University of Washington, Seattle, WA 98195} 
	\email{hoffman@math.washington.edu} 
\author[Hyojeong~Son]{ \ Hyojeong~Son} 
	\address{Department of Mathematics, University of Washington, Seattle, WA 98195} 
	\email{hjson@math.washington.edu}

\thanks{The authors were partially supported by NSF grant  DMS-1954059. H.S. was also partially supported by a fellowship from the University of Washington Math department.}

\begin{abstract} The Activated Random Walk (ARW) model is a promising candidate for demonstrating self-organized criticality due to its potential for universality. Recent studies have shown that the ARW model exhibits a well-defined critical density in one dimension, supporting its universality. In this paper, we extend these results by demonstrating that the ARW model on $\bZ$, with a single initially active particle and all other particles sleeping, maintains the same critical density. Our findings relax the previous assumption that required all particles to be initially active. This provides further evidence of the ARW model's robustness and universality in depicting self-organized criticality. 

\end{abstract}

\maketitle


\section{Introduction} 

In 1987, Bak, Tang, and Wiesenfeld introduced the concept of \textit{self-organized criticality} (SOC) \cite{BakTangWiesenfeld87}, a property observed in complex natural systems where energy accumulates slowly and is released intermittently. SOC describes systems that naturally evolve to a critical state without the need for fine-tuning parameters or external influences. Examples include financial markets, where variations in stock and commodity prices follow a power-law distribution, and forest fires, where accumulated flammable material can lead to large fires ignited by small sparks. These examples highlight the ubiquity and significance of SOC in various artificial and natural systems.

Since the concept of SOC was proposed, the search for a universal SOC model has led to extensive research on the \textit{deterministic sandpile model}. In this model, each vertex in a graph has a nonnegative number of chips, and a vertex can topple when the number of chips at that vertex equals or exceeds its degree. A toppling vertex distributes one chip to each neighboring vertex, which can trigger a chain of topplings. This deterministic model exhibits intricate fractals and patterns \cite{levine2016laplacian}. However, its deterministic nature limits its ability to exhibit certain critical behaviors \cites{dhar1999abelian,FeyLevineWilson10,JoJeong10}. Ideally, a universal model should exhibit universality in the sense that macroscopic properties are independent of microscopic details, making the system robust to perturbations.

The probabilistic variant of the deterministic sandpile model is the \textit{stochastic sandpile model}. In this model, instead of sending one chip to each neighboring vertex when a site topples, a fixed number of particles are sent to neighboring vertices chosen independently according to some probability distribution. There is evidence that this stochastic model features universality \cite{biham}. However, this model involves pairwise correlations in particle movements, which complicates the analysis of its dynamics.

More recently, the \textit{Activated Random Walk} (ARW) model has emerged as a promising candidate for a universal SOC model. At each site, the initial number of particles is i.i.d.\ and sampled from an ergodic distribution with mean $\rho$. In this continuous-time interacting particle system on $\mathbb{Z}$, each particle is either active or sleeping. Active particles perform symmetric random walks at rate $1$ and fall asleep at rate $\lambda \in (0,\infty)$. Sleeping particle remains stationary until awakened by an active particle and is instantly reactivated if other particles are present at the site. Since particles in the model perform random walks and fall asleep independently, the ARW model is more tractable than the stochastic sandpile model and has yielded significant findings.

One of primary questions in ARW dynamics is whether the system fixates or remains active, and under what conditions. Fixation occurs when every site is visited by active particles only finitely often, while the system stays active if there is ongoing particle movement and interaction. One might conjecture that the critical density, where a phase transition occurs, depends on the specific version of the ARW model being considered.

Recently, the second author, Junge, and Johnson \cite{HoffmanJohnsonJunge24} proved that for each $\lambda > 0$, the critical densities in the driven-dissipative model, the point-source model, the fixed-energy model on $\bZ$, and the fixed-energy model on the cycle all exist and are equal, denoted by $\crist$. This result confirms the \textit{density conjecture}, reinforcing the ARW model as a promising candidate for universality due to its well-defined critical density $\crist$.

Furthermore, the long-term behavior of a system depends on the particle density $\rho$ and the sleep rate $\lambda$. Specifically, Rolla, Sidoravicius, and Zindy demonstrated the following result, which highlights a phase transition based on these parameters:

\begin{thm}[Theorem 1 of {\cite{RollaSidoraviciusZindy19}}]\label{willow}
Consider the ARW model on $\mathbb{Z}^d$ for a fixed $d \geq 1$ with a given sleep rate $\lambda$. There exists a critical density $\crist$ such that for any spatially ergodic distribution $\nu$ supported on configurations where all particles are initially active, the following holds: a configuration sampled from $\nu$ almost surely fixates if the average density $\rho$ is less than $\crist$, and almost surely remains active if $\rho$ is greater than $\crist$.
\end{thm}

This theorem allows us to choose any shift-invariant ergodic measure for the initial configuration. However, there is an assumption that all particles should initially be active. In this paper, we weaken this assumption to show that even with only one active particle, there is a positive probability that the system remains active.

A very similar result was recently proved by Forien \cite{nicolas2024macroscopic}. He studied activated random walk on an interval where particles can leave through the endpoints. He showed that in the supercritical regime, a positive density of particles leaves the interval with positive probability. This paper significantly strengthens his result.

We consider a starting configuration $\omega \in \{ 0,\s\}^{\mathbb{Z}}$. In this configuration,  $\omega(i) = 0$ denotes that site $i$ is empty, and $\omega(i) = \s$ indicates that site $i$ contains a sleeping particle. The number of particles at each site is an independent Bernoulli random variable with mean $\rho \in (0,1)$, referred to as the particle density. 

Given a configuration $\omega \in \{0, \s\}^{\mathbb{Z}}$ we define a new configuration $\init$ by setting $\init(i) = \omega(i)$ for all $i \neq 0$ and placing one active particle at the origin, that is, $\init(0) = 1$. For any configuration $\nu \in \{ 0,\s, 1,2,\dots\}^{\mathbb{Z}}$, let $\mathbb{P}_{\nu}$ represent the probability measure of the system's evolution starting from $\nu$.

We are now in position to describe our main theorem.
Let $\fixation$ be the event that the system eventually fixates. Our goal is to prove the following theorem:

\begin{thm}\label{chipmunk}
Let $\omega \in \{ 0,\s\}^{\mathbb{Z}}$ be a configuration with particle density $\rho > \crist$, where all particles are initially asleep and the $\omega(i)$ are i.i.d. Bernoulli random variables with expectation $\rho$. For almost every $\init$,
$$0 < \mathbb{P}_{\init}(\text{\fixation}) < 1.$$
\end{thm}
The lower bound is easy while the upper bound is significantly more involved.

\begin{proof}[Proof of Lower Bound] At the beginning of our process, there is only one active particle at the origin. If this particle falls asleep before moving, then there are no active particles available to awaken any remaining sleeping particles. As a result, the system achieves \text{\fixation} immediately.

An active particle either jumps to a neighboring site at rate $1$ or falls asleep at rate $\lambda$. The total rate of these actions is $1 + \lambda$. Therefore, the probability that the system fixates is at least $\frac{\lambda}{1 + \lambda}$. \end{proof}

Theorem \ref{chipmunk} shows that the critical value is quite sharp. Below the critical value a configuration least likely to fixate (one with all particles initially active) still fixates. Above the critical value a configuration most likely to fixate (one with all but one particle initially asleep) still has a positive probability of not fixating.

\section{Site-wise Construction of ARW} \label{Section: Site-wise Construction of ARW}

In our proof of the upper bound, we utilize a site-wise representation for the ARW, where infinite stacks of instructions are attached to each site. This construction, underpinned by the abelian Property (Lemma \ref{camellia}), allows us to focus on the moves taken by particles at specific sites without concerning ourselves with the order in which they occur.

Consider the ordered set $\mathbb{N} \cup \{\s\}$ with $0 < \s < 1 < 2 < \cdots$. A \textit{configuration} is an element $\omega \in \{0, \s, 1, 2, 3, \dots\}^{\bZ}$, where $\omega(x)$ gives the state of site $x \in \bZ$. If $\omega(x) = 0$, there is no particle at site $x$; if $\omega(x) = \s$, there is one sleeping particle; and if $\omega(x) = n \geq 1$, there are $n$ active particles.

When an active particle visits a site with a sleeping particle, it wakes up the sleeping particle, resulting in $\s + 1 = 2$. If a site is given a sleep instruction, the effect depends on the number of active particles at the site. Precisely, if there is exactly one active particle at the site, the particle becomes sleeping, so the state changes from $1$ to $\s$. If there are $n \geq 2$ active particles at the site, the sleep instruction has no effect, and the state remains at $n$. We define a site $x$ to be \textit{stable} if $\omega(x) < 1$ (i.e., if there are no active particles at $x$), and \textit{unstable} otherwise. By defining $\lvert \s \rvert = 1$, we denote the number of particles at site $x$ by $\lvert \omega(x) \rvert$.

Each site $x \in \bZ$ has an infinite \textit{stack} of instructions $\Instr_x(k)$, where $k$ denotes the $k$-th instruction at site $x$. The distribution of $\Instr_x(k)$ is:
\begin{align*}
  \Instr_x(k) = \begin{cases} 
    \Left & \text{with probability $\frac{1}{2(1+\lambda)}$,} \\
    \Right & \text{with probability $\frac{1}{2(1+\lambda)}$,} \\
    \Sleep & \text{with probability $\frac{\lambda}{1+\lambda}$}.
  \end{cases}
\end{align*}
Each $\Instr_x(k)$ is independent for all $x \in \bZ$ and $k \in \mathbb{N}^{+}$. And the rate at which we execute unused instructions at each site $x$ is given by $\mathds{1}_{\omega(x) \neq \s} \lvert \omega(x) \rvert (1+\lambda)$. The $\Left$ (respectively, $\Right$) instruction subtracts $1$ from $\omega(x)$ and adds $1$ to $\omega(x-1)$ (respectively, $\omega(x+1)$). Executing a $\Sleep$ instruction at a site with exactly one active particle changes $\omega(x) = 1$ to $\s$; if there are multiple active particles at the site, the sleep instruction has no effect. We can \textit{topple} a single site $x$ by executing its first unused instruction. Alternatively, we can \textit{topple} a sequence of sites $\alpha = (x_1, x_2, \ldots, x_l)$ by starting with $x_1$ and proceeding to $x_l$.

To track the number of topplings at each site, we introduce an \textit{odometer} function $U$: 
$$U: \mathbb{Z} \to \mathbb{N},$$ 
where $U(x)$ denotes the number of times site $x$ has been toppled. Additionally, $U_\alpha(x)$ represents the number of times site $x$ appears in the sequence $\alpha$. When toppling site $x$, we update the configuration $\omega$ and the odometer reading $U$ by defining the toppling operation acting on $(\omega, u)$ as: 
$$\Phi_x(\omega, U) = (\Instr_x(U(x) + 1)(\omega), U + \delta_x),$$
where $\delta_x$ is a function on $\mathbb{Z}$ that takes the value $1$ at $x$ and $0$ elsewhere. We say that $\Phi_x$ is \textit{legal} if site $x$ is unstable. A toppling sequence $\Phi_\alpha$ is a \textit{sequence of legal topplings} if, for each $1 \leq i \leq l$, the toppling $\Phi_{x_i}$ is legal after performing the preceding topplings $\Phi_{x_{i-1}} \cdots \Phi_{x_2} \Phi_{x_1}$. We call $\alpha$ \textit{stabilizing} if, after performing all topplings in the sequence $\alpha$, no unstable sites remain.

The abelian property is an essential feature of the ARW, enabling us to utilize the toppling procedure described in Section~\ref{section_good}, where we stabilize the system one interval at a time. We formalize this property in the following lemma:

\begin{lemma}[Abelian Property {\cite{rolla2020activated}}]\label{camellia}
   Let $\alpha$ and $\beta$ be two legal sequences of topplings of a configuration $\omega$. Suppose $U_\alpha(x) = U_\beta(x)$ for all $x \in \mathbb{Z}$. Then their final configurations are identical, that is,
   $$\Phi_\alpha(\omega) = \Phi_\beta(\omega).$$
\end{lemma}

We will also make use of a particle-wise representation of ARW. In this representation, each particle has its own stack of i.i.d.\ instructions. The stacks for each particle are also independent. By the abelian property, at any time we can choose to move any active particle. We do this by executing the next unused instruction for that active particle. If this particle-wise process terminates for a given initial configuration, then it will have the same final distribution on the location of particles and odometers as the site-wise representation has.


\section{Good Configurations}
\label{section_good}

\subsection{Toppling Procedure}
By the abelian property (Lemma \ref{camellia}), the stabilization of the system is independent of the order in which sites are toppled. Therefore, we can perform the following inductive partial stabilization procedure, starting from the configuration $\init$ with one active particle at the origin:

\begin{enumerate}
    \item Given $\rho$, choose $\delta<(\rho-\crist)/3$ and a sufficiently large integer  $k>200/\delta^2$. 
    \item For each positive integer $n$,  consider the interval $I_n = [ -k^n, k^n]$.
    \item For each $n$, proceed inductively as follows:
    \begin{enumerate}
        \item Evolve the system within the interval $I_n$, freezing active particles at the boundaries of $I_n$, until no active particles remain in the interior of $I_n$.
        \item If active particles are present at the boundaries of $I_n$, extend the interval to $I_{n+1}$ and repeat the process.
    \end{enumerate}
\end{enumerate}

Throughout the remainder of this paper, when we refer to the odometer, we are considering the odometer associated with the stabilization step on the interval $I_{n+1}$. Specifically, the stabilization begins with the configuration $\omegaprime$ and concludes with the configuration $\omegadouble$ on $I_{n+1}$. Also, the $n$th step stabilization refers to our toppling procedure for stabilizing the interval $I_n$. Let $\positive$ denote the subset of $(-k^{n+1}, k^{n+1})$ where the odometer is positive during this stabilization step.

\begin{figure}[h]
    \centering
\begin{tikzpicture}[scale=1.9]


    \def\numTicks{53}

    \def\xStart{-3.3}
    \def\xEnd{3.3}
    \def\xStep{(\xEnd - \xStart)/(\numTicks - 1)}

	\draw[<-,>=latex, line width=0.5pt] ({\xStart },0) -- ({\xStart + \xStep*7},0);
   \draw[->,>=latex,line width=0.5pt] ({\xEnd - \xStep*7},0) -- ({\xEnd},0);

\draw[line width=0.5pt] ({\xStart + \xStep*9},0) -- ({\xStart + \xStep*25},0);

\draw[line width=0.5pt] ({\xEnd - \xStep*25},0) -- ({\xEnd - \xStep*9},0);

  \node at ({\xStart + \xStep*8}, -0.01) {\scalebox{1}{$\mathbf{\cdots}$}};

    \node at ({\xStart + \xStep*26}, -0.01) {\scalebox{1}{$\mathbf{\cdots}$}};
  
   \node at ({\xEnd - \xStep*8}, -0.01) {\scalebox{1}{$\mathbf{\cdots}$}};

\foreach \i in {2,...,50} {
    \ifnum \i=8
        \else\ifnum \i=26 
            \else\ifnum \i=44
                \else
                    \pgfmathsetmacro{\x}{\xStart + \xStep*\i}
                    \draw[thick, line width=0.5pt] (\x,0.04) -- (\x,-0.04);
            \fi
        \fi
    \fi
}

    \foreach \i in {2,3,5,10,12,14,15,17,20,23,28,31,33,35,36,39,40,41,48} {
        \pgfmathsetmacro{\x}{\xStart + \xStep*\i}
        \fill[blue] (\x,0.1) circle (0.7pt);
    }

    \foreach \y in {0.1,0.4,0.47,0.54,0.61,0.68} {
        \pgfmathsetmacro{\x}{\xStart + \xStep*26-\xStep*4}
        \fill[red] (\x,\y) circle (0.7pt);
    }
    \node at ({\xStart + \xStep*26-\xStep*4}, 0.30) {\scalebox{1.0}{$\mathbf{\vdots}$}}; 
  
       \foreach \y in {0.1,0.4,0.47,0.54} {
        \pgfmathsetmacro{\x}{\xStart + \xStep*26+\xStep*4}
        \fill[red] (\x,\y) circle (0.7pt);
    }
    \node at ({\xStart + \xStep*26+\xStep*4}, 0.30) {\scalebox{1.0}{$\mathbf{\vdots}$}};

    \pgfmathsetmacro{\xstart}{\xStart + \xStep*26-\xStep*4}
    \pgfmathsetmacro{\xend}{\xStart + \xStep*26+\xStep*4}
    \draw [decorate,decoration={brace,mirror,amplitude=8pt},yshift=-0.07cm]
        (\xstart,0) -- (\xend,0) node [black,midway,yshift=-0.53cm] {\footnotesize\(I_n\)};

    \begin{scope}[shift={(0,-1.2)}] 

    \def\numTicks{53}

    \def\xStart{-3.3}
    \def\xEnd{3.3}
    \def\xStep{(\xEnd - \xStart)/(\numTicks - 1)}

	\draw[<-,>=latex, line width=0.5pt] ({\xStart },0) -- ({\xStart + \xStep*7},0);
   \draw[->,>=latex,line width=0.5pt] ({\xEnd - \xStep*7},0) -- ({\xEnd},0);

\draw[line width=0.5pt] ({\xStart + \xStep*9},0) -- ({\xStart + \xStep*25},0);

\draw[line width=0.5pt] ({\xEnd - \xStep*25},0) -- ({\xEnd - \xStep*9},0);

  \node at ({\xStart + \xStep*8}, -0.01) {\scalebox{1}{$\mathbf{\cdots}$}};

    \node at ({\xStart + \xStep*26}, -0.01) {\scalebox{1}{$\mathbf{\cdots}$}};
  
   \node at ({\xEnd - \xStep*8}, -0.01) {\scalebox{1}{$\mathbf{\cdots}$}};

\foreach \i in {2,...,50} {
    \ifnum \i=8
        \else\ifnum \i=26 
            \else\ifnum \i=44
                \else
                    \pgfmathsetmacro{\x}{\xStart + \xStep*\i}
                    \draw[thick, line width=0.5pt] (\x,0.04) -- (\x,-0.04);
            \fi
        \fi
    \fi
}

    \foreach \i in {2,3,5,10,12,20,27,29,36,41,48} {
        \pgfmathsetmacro{\x}{\xStart + \xStep*\i}
        \fill[blue] (\x,0.1) circle (0.7pt);
    }

    \foreach \y in {0.1,0.4,0.47,0.54,0.61,0.68,0.98} {
        \pgfmathsetmacro{\x}{\xStart + \xStep*26-\xStep*23}
        \fill[red] (\x,\y) circle (0.7pt);
    }
    \node at ({\xStart + \xStep*26-\xStep*23}, 0.30) {\scalebox{1.0}{$\mathbf{\vdots}$}}; 

      \node at ({\xStart + \xStep*26-\xStep*23}, 0.88) {\scalebox{1.0}{$\mathbf{\vdots}$}}; 
  
       \foreach \y in {0.1,0.4,0.47,0.54,0.84,0.91,0.98 } {
        \pgfmathsetmacro{\x}{\xStart +  \xStep*26+\xStep*23}
        \fill[red] (\x,\y) circle (0.7pt);
    }
    \node at ({\xStart +  \xStep*26+\xStep*23}, 0.30) {\scalebox{1.0}{$\mathbf{\vdots}$}};
     \node at ({\xStart +  \xStep*26+\xStep*23}, 0.74) {\scalebox{1.0}{$\mathbf{\vdots}$}};

    \pgfmathsetmacro{\xstart}{\xStart + \xStep*26-\xStep*23
}
    \pgfmathsetmacro{\xend}{\xStart + \xStep*26+\xStep*23}
    \draw [decorate,decoration={brace,mirror,amplitude=8pt},yshift=-0.07cm]
        (\xstart,0) -- (\xend,0) node [black,midway,yshift=-0.53cm] {\footnotesize\(I_{n+1}\)};
    \end{scope}

\end{tikzpicture}
\caption{Example of the configurations $\omegaprime$ (top) and $\omegadouble$ (bottom) during the stabilization of the interval $I_{n+1}$ according to the toppling procedure. Red circles represent active particles, while blue circles represent sleeping particles.}
\end{figure}
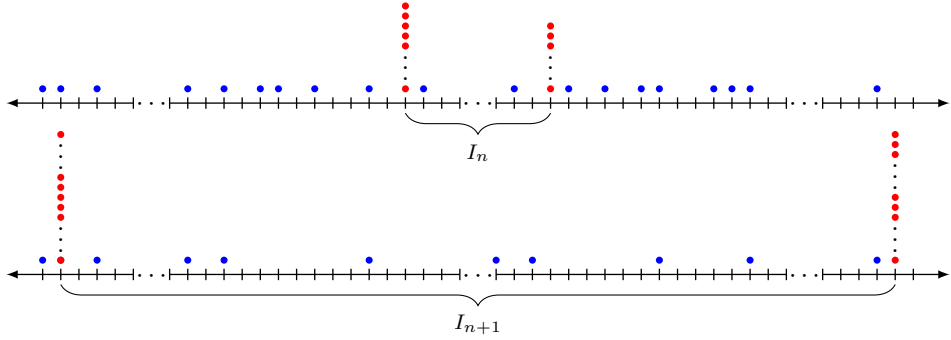

\subsection{Good Configurations}


Our goal is to show that, for a `typical' configuration in our stabilization process, the probability that the system stabilizes without active particles at the boundaries of $I_{n+1}$ is very small, given that it satisfied the previous stabilization step within the interval $I_n$. However, there exists a class of configurations in which no active particles remain at the boundaries after stabilization. To address this, we introduce the concept of a \textit{good} configuration. Let $\gamma < \delta/10$ be chosen such that both $k\gamma$ and $1/\gamma$ are integers.


\begin{dfn}
    A configuration $\omega$ is \textit{good} for a positive integer $n$ if $\omega$ satisfies the following criteria for the $n$th stabilization step:
    \begin{enumerate}[label=Criterion \arabic*, leftmargin=6.0em, labelwidth=3.8em, labelsep=0.5em, itemsep=1em]
        \item\label{c1} The interval $I_n$ contains at least $2k^n(\crist + 2\delta)$ particles.
        \item\label{c2} Every subinterval of length $\gamma k^n$ within $I_n\setminus\{-k^{n},k^n\}$ contains at most $\gamma k^n (\crist + \delta)$ particles.
    \end{enumerate}
\end{dfn}

\subsection{Prior Results}

To estimate the probability that the system fails to satisfy the criteria for a good configuration at step $n+1$, given that it satisfies the criteria at step $n$, we use the following recent theorem.

\begin{thm}[Theorem 8.4 of \cite{HoffmanJohnsonJunge24}]\label{peanut butter}
Let $n$ be a positive integer. Consider activated random walk on $\bZ$ with sleep rate $\lambda > 0$. Let $\sigma$ be an initial configuration with no sleeping particles on the interval $[a, a + n - 1]$ for some integer $a$. Let $X_n$ be the number of particles left sleeping in the stabilization of $\sigma$ on $[a, a + n - 1] $. For any $\delta > 0$,
$$
\bP_{\sigma}[X_n \geq (\crist+\delta) n] \leq C e^{-c n},
$$
where $C$ and $c$ are positive constants depending only on $\lambda$ and $\delta$.
\end{thm}

In particular, we will use the following variant of the theorem.

\begin{cor}\label{daffodil} 
Let $n$ be a positive integer. Consider activated random walk on $\mathbb{Z}$ with sleep rate $\lambda > 0$. Let $\sigma$ be an initial configuration with at least one active particle on the interval $[a, a + n - 1]$ for some integer $a$.

Let $b$ be an integer and $\gamma \in (0,1]$ such that $I_{b,\gamma} = [b, b +  n\gamma - 1]\cap \bZ \subseteq I$. Let $U$ be a stable odometer on $I_{b,\gamma}$ that satisfies $U(b) = u_0$ and $U(x) > 0$ for every site $x \in I_{b,\gamma}$ containing a particle in $\sigma$. Define $Y_{I_{b,\gamma}}(u_0)$ to be the number of sleeping particles remaining in $I_{b,\gamma}$ after applying $U$. Then, for any $\rho > \crist(\lambda)$,
$$\mathbb{P}_\sigma\left[ Y_{I_{b,\gamma}}(u_0) \geq \rho  n\gamma  \right] \leq C e^{-c n}$$
for some constants $C$ and $c$, depending only on $\lambda$ and $\rho$.

\end{cor}

\begin{proof}
This is a restatement of
 equation (81) in the proof of Theorem 8.4 in \cite{HoffmanJohnsonJunge24} (plus an application of the union bound).    
\end{proof}

\subsection{Transition Probabilities}

Let $G_n$ denote the set of all good configurations on the interval $I_n$ at step $n$. We define the event
\begin{align*}
D_{n+1} &= \text{the resulting configuration on } I_{n+1} \text{ is not good or the odometer is} \\
        &\quad \text{somewhere equal to zero}.
\end{align*}

Our goal is to analyze the probability of the event $D_{n+1}$ conditioned on  configurations $\omegatilde \in G_n$.
With a slight abuse of notation for $\omegatilde \in G_n$ we write 
$\bP_{\omegatilde}$ to indicate the randomness of the distribution of particles in $I_{n+1}$ at the end of the stabilization of $I_{n+1}$.  To achieve this, we partition $D_{n+1}$ into four distinct cases, denoted $D_{n+1}^{1}$, $D_{n+1}^{2}$, $D_{n+1}^{3}$, and $D_{n+1}^{4}$, based on the following considerations:

\begin{enumerate}
    \item \textbf{Odometer Considerations:}
    \begin{itemize}
\item     \textbf{Case $D_{n+1}^{1}$}: The sum of the odometers in $I_{n+1}$  exceeds $k^{3.5(n+1)}$.
       \item \textbf{Case $D_{n+1}^{2}$}: The sum of the odometer does not exceed $k^{3.5(n+1)}$, and the odometer is somewhere equal to zero.
    \end{itemize}
    
    \item \textbf{Criterion Considerations:} Identify which of the two criteria the configuration fails to satisfy.
    \begin{itemize}
        \item \textbf{Case $D_{n+1}^{3}$}: The sum of the odometer does not exceed $k^{3.5(n+1)}$, and the configuration fails \ref{c1}. 
        \item \textbf{Case $D_{n+1}^{4}$}: The sum of the odometer does not exceed $k^{3.5(n+1)}$, the odometer is everywhere positive, and the configuration fails \ref{c2}.
         \end{itemize}
\end{enumerate}

This decomposition allows us to systematically analyze and bound the probability of $D_{n+1}$ by considering each of these four cases separately.

\begin{lemma}\label{saturday}
    $$D_{n+1} \subset \bigcup_{i=1}^4 D_{n+1}^i.$$
\end{lemma}

\begin{proof}
Consider the case where the sum of the odometer readings in $I_{n+1}$ exceeds $k^{3.5(n+1)}$. This scenario is captured by $D_{n+1}^1$. Suppose, instead, that the sum does not exceed $k^{3.5(n+1)}$. In this situation, either the odometer is zero at some location in $I_{n+1}$, leading to $D_{n+1}^2$, or the odometer is positive everywhere but the configuration is not good. In the latter case, the configuration must fail either condition \ref{c1}, which implies $D_{n+1}^3$, or condition \ref{c2}, which implies $D_{n+1}^4$.
\end{proof}

\begin{lemma}\label{sunday}
For positive constants $C$ and $c$ depending on $\lambda$, we have
$$
\sup_{\omegatilde \in G_n} \bP_{\omegatilde}\left(D_{n+1}^1\right) \leq C e^{-c k^{0.5 (n+1) }}.
$$
\end{lemma}


\begin{proof}
We use a particle-wise representation of ARW. There are at most $2 k^{n+1} + 1 \leq 3 k^{n+1}$ particles in $I_{n+1}$. If the sum of the odometers in $I_{n+1}$ exceeds $k^{3.5(n + 1)}$, then there must exist at least one particle that executes at least $(1/3) k^{2.5(n+1)}$ instructions (either steps or sleep instructions). 

Each time a particle takes $k^{2 n + 2}$ steps, it has a probability of at least $c'(\lambda) > 0$ of exiting $I_{n+1}$. Therefore, the probability that a particle takes more than $k^{2.5(n+1)}$ steps without reaching the boundaries of $I_{n+1}$ decays exponentially in $k^{0.5(n + 1)}$.

Applying the union bound over all particles in $I_{n+1}$, we obtain the desired result.
\end{proof}

\begin{lemma}\label{lays}
There exist constants $C$ and $c$ such that
$$
\sup_{\omegatilde \in G_n} \bP_{\omegatilde}\left(D_{n+1}^3\right) \leq C e^{- c k^{n+1}}.
$$
\end{lemma}

\begin{proof}
Let $X$ be the number of particles in $I_{n+1}$. Initially, each site contains one sleeping particle, following a Bernoulli distribution with mean $\rho$. In our toppling procedure at step $n+1$, we ensure that no particle escapes from $I_{n+1}$. Thus, the expectation of $X$ is $\rho(2k^{n+1} + 1)$.

We consider the event that $X$ is less than $2k^{n+1} (\crist + 2\delta)$. Since $\rho > \crist + 3\delta$, this value is less than the expected value by a fixed proportion of $\delta$. By the Chernoff bound, the probability that $X$ deviates below its mean by a fixed fraction decays exponentially in $k^{n+1}$. 
\end{proof}

\begin{lemma} \label{sonic}
There exist positive constants $C$ and $c$ such that
$$
\sup_{\omegatilde \in G_n} \bP_{\omegatilde}\left(D_{n+1}^4\right) \leq C e^{- c k^{n+1} }.
$$ 
\end{lemma}

\begin{proof}
For $D_{n+1}^{4}$ to occur, there must exist a subinterval of length $\gamma k^{n+1}$ with a particle density exceeding $\crist + \delta$ after the $(n+1)$th stabilization. There are at most $2k^{n+1}$ choices of subinterval and $k^{3.5(n+1)}$ choices of odometer at the leftmost edge of the subinterval. By Corollary \ref{daffodil}, the probability of such an event for each of these choices decays exponentially in $k^{n+1}$. Applying the union bound over all possible subintervals and choices of the odometer we obtain the desired result.
\end{proof}

\section{Completing the Proof}

We will break up the event $D_{n+1}^{2}$ into pieces depending on the location of the leftmost and/or rightmost site with a zero odometer.

\begin{enumerate}
    \item $F_{n+1}^1$:
    {$D^2_{n+1}$ occurs and there is no zero odometer $\geq k^n$,}
    \item $F_{n+1}^2$: $D^2_{n+1}$ occurs and there is no zero odometer $\leq -k^n$ and
    \item $F_{n+1}^3$: 
   \end{enumerate}
In the case that $F_{n+1}^3$ occurs we define $z$ to be the smallest value with positive odometer and $z'$ to be the largest value with positive odometer. Then we can partition 
$$F_{n+1}^3 =\bigcup_{z',z} F_{n+1}^{3,z',z}.$$
\begin{lemma}\label{ipad}
    $$ D^2_{n+1} \subset F_{n+1}^1 \cup F_{n+1}^2 
    \cup \left( \bigcup_{z,z'} F_{n+1}^{3,z,z'} \right). $$
\end{lemma}

\begin{proof}
If $D^2_{n+1}$ occurs and $F_{n+1}^1$ and $F_{n+1}^2$ do not occur then $F_{n+1}^3$ must occur. Then there are $z$ and $z'$ such that $F_{n+1}^{3,z,z'}$ occurs.
\end{proof}

To bound the probabilities of these events, recall that for a configuration $\omega$, $\omega(i)$ represents the number of particles at site $i$. Moreover, we define that if there is a sleeping particle at site $i$, then $\omega(i) = 1$. Using this description, our next definition will characterize the environment in $I_{n+1}\setminus I_{n}$ before stabilizing.

\begin{dfn}\label{plentiful}
We say that $\omegaprime \in \{0,1,2,\dots\}^{\{ - k^{n+1}, \dots , k^{n+1} \}}$ is \plentiful\ if, for every $j \in \{1, 2, \dots, (k-1)/\gamma \}$, the following inequalities hold:
$$\sum_{i = k^n+(j-1) \gamma k^n +1}^{k^n+ j \gamma k^n} \omegaprime(i) \geq \gamma k^n (\crist + 2\delta)$$ and $$\sum_{i = k^n+(j-1) \gamma k^n +1}^{k^n+ j \gamma k^n} \omegaprime(-i) \geq \gamma k^n (\crist + 2\delta).$$
\end{dfn}
The next definition describes the configuration that we expect to see after stabilization in the region $\positive$.

\begin{dfn}\label{sparse}
We say that $\omegadouble \in \{0,1,2,\dots\}^{\{-k^{n+1}, \dots, k^{n+1}\}}$ is \sparse\ in  $R \subset  \Z$ if, for every integer $j$ such that $[j+1,j+ \gamma k^n]\subseteq R$, the following inequality holds:
$$\sum_{i = j +1}^{j+\gamma k^n} \omegadouble(i) \leq \gamma k^n (\crist + \delta).$$
\end{dfn}

Next we bound the probabilities that the initial and final configurations are atypical.
\begin{lemma} \label{plenty}
There exist positive constants $C$ and $c$ such that
$$
\sup_{\omegatilde \in G_n} \bP_{\omegatilde} 
\left( \omegaprime \text{ is not \plentiful} \right) \leq C  e^{- c k^n }.
$$
\end{lemma}

\begin{proof}
In each interval, the total number of particles is the sum of $\gamma k^n$ independent Bernoulli random variables with mean $\rho$, where $\rho > \crist + 3\delta$. The probability that the total number of particles in an interval is less than $\gamma k^n (\crist + 2\delta)$ corresponds to a deviation below the expected value by at least $\gamma k^n (\rho - \crist - 2\delta) \geq \gamma k^n \delta$. Applying the Chernoff bound, the probability of such a deviation is exponentially small in $k^n$. Since there are $2(k-1)/\gamma$ intervals, the union bound implies that the probability that $\omegaprime$ is not \plentiful\ is $2(k-1)/\gamma$ times an exponentially decreasing term.
\end{proof}

\begin{lemma}\label{spicy}
There exist positive constants $C$ and $c$ such that
$$
\sup_{\omegatilde \in G_n} \bP_{\omegatilde} 
\left( \omegadouble\text{ is not \sparse\ in } \positive\right) \leq C  e^{- c k^n }.
$$
\end{lemma}

\begin{proof}
    This follows by the same argument used in Lemma \ref{sonic}.
\end{proof}

\begin{lemma}\label{paris}
    Let $\omegaprime,\ \omegadouble \in \{0,1,2,\dots\}^{\{-k^{n+1},\dots , k^{n+1}\}}$ be the initial and final configurations of the stabilization in $I_{n+1}$. If
    \begin{enumerate}
         \item there exist $z \leq  {k^n}$ such that $z \notin \positive,$
        \item $\omegaprime$ is \plentiful  \ and
        \item$\omegadouble$ is \sparse\ in $[z,k^{n+1}),$ 
    \end{enumerate}
    then 
       \begin{equation} 
           \sum_{i=z}^{k^{n+1}}i\ (\omegadouble(i)-\omegaprime(i))>(\delta/10)(k^{n+1})^2. \label{ugly} \end{equation}
\end{lemma}

\begin{proof}
We will construct a pairing between the particles in $\omegaprime$ and $\omegadouble$ that satisfies 
    \begin{enumerate}
       \item at least $(\delta/2)k^{n+1}$ particles in $\omegaprime$ are paired with a particle in $\omegadouble$ that is more than $(1/3)k^{n+1}$ to its right. 
        \item no particle in $\omegadouble$ is paired with a particle in $\omegaprime$ that is more than $(3/\delta)k^n$ to its left,        
        \item at most $(3/\delta)k^n$ particles in $\omegadouble$ is paired with a particle in $\omegaprime$ that is more than $\gamma k^n$ to its left,
\end{enumerate}
These conditions  imply that the sum on the left hand side of \eqref{ugly} is at least
\begin{eqnarray*}
(\delta/2) k^{n+1}(1/3) k^{n+1} -((3/\delta)k^{n})^2-(\gamma k^n)(2k^{n+1})
& \geq & (k^{2n})(k^2\delta/6 -9/\delta^2 -2\gamma k) \\ 
& \geq & (k^{2n})(k^2\delta/10).
\end{eqnarray*}
The last inequality holds for sufficiently large $k$. Thus \eqref{ugly} holds.

Since $\omegaprime$ is \plentiful\ and $\omegadouble$ is \sparse\ in $[z,k^{n+1}]$, in $[-k^n,k^n +(2/\delta)k^n]$ there are more particles in $\omegaprime$ than in $\omegadouble$. Thus we can pair every particle in this region in $\omegadouble$ to a particle in the same region in $\omegaprime.$ This interval has length at most $(3/\delta)k^n$ and it has at most $(3/\delta)k^n$ particles in $\omegadouble.$

In the region from 
$$[k^n +(2/\delta)k^n,k^{n+1}/2]\supset[k^{n+1}/10,k^{n+1}/2]$$ 
there are at least $\gamma k^{n+1}/3$ particles in $\omegaprime$ that are paired with a particle at $k^{n+1}.$ Thus we can create the desired pairing.
\end{proof}

\begin{lemma}
    \label{united}
Let $U_{n+1}(x)$ be the odometer at $x$ in the stabilization 
at stage $n+1.$
Let $\tilt$ be the event that
    \begin{enumerate}
      \item $\sum_{x=-k^{n+1}}^{k^{n+1}} U_{n+1}(x) \leq k^{3.5(n+1)}$ and
      \item there exists $z \not\in \positive$ with
      $$\sum_{j=z}^{k^{n+1}}j(\omegadouble(j)-\omegaprime(j))>(\delta/10)k^{2n+2}.$$
    \end{enumerate} There exists $C$ and $c$
    $$\sup_{\omegatilde \in G_n}\bP_{\omegatilde}(\tilt)\leq C e^{-ck^{.5n}}$$
\end{lemma}
\begin{proof}
    If $z \notin \positive$ then all of the particles that start to the right of $z$ remain to the right of $z$ throughout the stabilization process. Choose an order of updating the particles and call the configurations $$\omega_0=\omegaprime, \omega_1,\dots, \omega_J =\omegadouble.$$
    Then
    $$X(j)=\sum_{x=z}^{k^{n+1}}x(\omega_{j}(x)-\omega_{j-1}(x))$$
 for $j =1,2 , \dots, J$ is a martingale where the step size is bounded by 1. Also note that $$J\leq \sum_{x=-k^{n+1}}^{k^{n+1}} U_{n+1}(x) \leq k^{3.5(n+1)}. $$
    By the Azuma-Hoeffding inequality we have that 
    $$\bP(\tilt)\leq C e^{-ck^{.5n}}.$$
\end{proof}

\begin{lemma} \label{reacts}
There exists $C$ and $c$ such that
   $$\sup_{\omegatilde \in G_n} \bP_{\omegatilde}   F_{n+1}^{1} \leq C e^{-ck^{n}} \ \ \ \ \text{ and }  \ \ \ \ \sup_{\omegatilde \in G_n} \bP_{\omegatilde}   F_{n+1}^{2} \leq C e^{-ck^{n}}.$$ 
\end{lemma}
   
\begin{proof}
Suppose that $F_{n+1}^{1}$ occurs and $\omegaprime$ is $\plentiful$. And let $z$ be the largest site with zero odometer. Then we have that $z$ is less than $k^n$.

 For the first inequality by Lemma \ref{paris} these events only happen if $\omegaprime$ is not $\plentiful$, $\omegadouble$ is not $\sparse$ to the right of $z'$ or \eqref{ugly} holds (the center of mass changes a lot.) The probabilities of these events are bounded by Lemmas \ref{plenty}, \ref{spicy} and \ref{united}.
 
The second inequality follows by symmetry.
\end{proof}

\begin{lemma} \label{hazel}
There exists $C$ and $c$ such that
    $$\sup_{\omegatilde \in G_n} \bP_{\omegatilde} \left( \bigcup_{z,z'} F_{n+1}^{3,z,z'} \right)\leq C e^{-ck^{n+1}}.$$
\end{lemma}
   \begin{proof}
First we deal with the case that the set of positive odometers is connected.
As $\omegatilde \in G_n$ if $\omegaprime$ is $\plentiful$ then for every $z<k^n$ and $z'>k^n$ there are at least 
$$(z'-z)(\crist +2\delta)-2\gamma k^n$$
particles in $[z,z']$.
If $F_{n+1}^{3,z,z'}$ occurs then all of these particles are in $[z,z']$ after stabilization.
If $\omegadouble$ is $\sparse$ in $(z,z')$ then there are at most 
\begin{eqnarray*}
(z'-z)(\crist +\delta)+2\gamma k^n
& = & (z'-z)(\crist +2\delta)-\delta(z'-z)+2\gamma k^n \\ 
&\leq & (z'-z)(\crist +2\delta)-2\gamma k^n-\delta(2k^n)
+4\gamma k^n \\ 
& <&(z'-z)(\crist +2\delta)-2\gamma k^n
\end{eqnarray*}
particles in $[z,z']$.
Thus $F_{n+1}^{3,z,z'}$ only happens if $\omegaprime$ is not $\plentiful$ or $\omegadouble$ is not sparse in $[z,z']$. 

If the set of positive odometers is not connected then the set $\positive$ is of the form $[z,z_1] \cup [z_2,z']$ for some $z_1$ and $z_2$. A very similar computation to the one above again shows that 
either $\omegaprime$ is not $\plentiful$ or $\omegadouble$ is not sparse in $\positive$.
Combining Lemmas \ref{plenty} and \ref{spicy} and summing up over all choices of $z,z',z_1$ and $z_2$ completes the proof.
   \end{proof}

\begin{lemma}\label{hichew}
There exist $C$ and $c$ such that
$$
\sup_{\omegatilde \in G_n} \bP_{\omegatilde}\left(D_{n+1}^2\right) \leq C e^{-c k^{n^{0.5}}}.
$$
\end{lemma}

\begin{proof}
By Lemmas \ref{ipad}, \ref{reacts}, and \ref{hazel} and the union bound.

\end{proof}

\begin{lemma}\label{trumpet}
There exist positive constants $C$ and $c$ such that
$$
\sup_{\omegatilde \in G_n} \bP_{\omegatilde}\left(D_{n+1}\right) \leq C e^{-c k^{n^{0.5}}}.
$$ 
\end{lemma}

\begin{proof}
By Lemmas \ref{sunday}, \ref{hichew}, \ref{lays}, and \ref{sonic}, the probabilities of the events $D_{n+1}^{1}$, $D_{n+1}^{2}$, $D_{n+1}^{3}$, and $D_{n+1}^{4}$ decay exponentially in $k^{n+1}$ or $k^{0.5 (n+1)}$. Therefore, by Lemma \ref{saturday}
there exist positive constants $C$ and $c$ such that
$$
\sup_{\omegatilde \in G_n} \bP_{\omegatilde}\left(D_{n+1}\right) \leq C e^{-c k^{n^{0.5}}}.
$$
\end{proof}

\begin{proof}[Proof of Upper Bound of Theorem \ref{chipmunk}]
    From Lemma \ref{trumpet} and the Borel-Cantelli lemma we get that the probability that $D_{n}$ does not occur for any $n$ is positive. By the definition of $D_n$ if  $D_n$ does not occur then every particle in $I_n$ moves at least once in the $n$th stabilization stage. Thus if $D_n$ does not occur for any $n$ then every particle moves infinitely often. Thus the system does not stabilize.
\end{proof}

\bibliographystyle{amsalpha}
\bibliography{main}

\end{document}